\documentclass[11pt, twoside]{article}
\usepackage{amsfonts,amssymb,amsmath,amsthm}
\usepackage{graphicx}
\usepackage[all]{xy}
\usepackage{multirow} 
\usepackage{psfrag,xmpmulti,amscd,color,pstricks, import}
\usepackage{setspace}
\usepackage{enumitem}

 \divide\oddsidemargin by 2
\newtheorem{thm}{Theorem}[section]
\newtheorem*{thm*}{Theorem}

\newtheorem{lem}[thm]{Lemma}

\theoremstyle{definition}
\newtheorem*{defn*}{Definition}

\newtheorem*{rem*}{Remark}

\numberwithin{equation}{section}

\definecolor{OrangeRed}{cmyk}{0,0.6,1,0}            
\definecolor{DarkBlue}{cmyk}{1,1,0,0.20}
\definecolor{DarkGreen}{cmyk}{1,0,0.6,0.2}
\definecolor{myblue}{rgb}{0.66,0.78,1.00}
\definecolor{Violet}{cmyk}{0.19,0.98,0.1,0}
\definecolor{Lavender}{cmyk}{0,0.48,0,0}

\renewcommand{\epsilon}{\varepsilon}
\renewcommand{\phi}{\varphi}

\title{Dynamics inside Fatou sets in higher dimensions}

\vspace{5cm}
\author{   
	Mi Hu \\}

\begin{document}

	\maketitle  
	\begin{abstract}
		In this paper, we investigate the behavior of orbits inside attracting basins in higher dimensions. Suppose $F(z, w)=(P(z), Q(w))$, where $P(z), Q(w)$ are two polynomials of degree $m_1, m_2\geq2$ on $\mathbb{C}$,
		$P(0)=Q(0)=0,$ and $0<|P'(0)|, |Q'(0)|<1.$  Let $\Omega$ be the immediate attracting basin of $F(z, w)$. 
		Then there is a constant $C$ such that for every point $(z_0, w_0)\in \Omega$, there exists a point $(\tilde{z}, \tilde{w})\in \cup_k F^{-k}(0, 0), k\geq0$ so that $d_\Omega\big((z_0, w_0), (\tilde{z}, \tilde{w})\big)\leq C,  d_\Omega$ is the Kobayashi distance on $\Omega$. However, for many other cases, this result is invalid. 
		\\
		
		

	\end{abstract}

	\section{Introduction}\label{sec1}
		In discrete dynamical systems, we are interested in  qualitatively and quantitatively describing the possible dynamical behavior under the iteration of maps satisfying certain conditions. In our paper \cite{RefH1}, Hu studied the dynamics of holomorphic polynomials on attracting basins and obtained Theorem \ref{thmA}:
		
		\begin{thm}\label{thmA} [Hu, 2022]
			Suppose $f(z)$ is a polynomial of degree $N\geq 2$ on $\mathbb{C}$, $p$ is an attracting fixed point of $f(z),$ $\Omega_1$ is the immediate basin of attraction of $p$, $\{f^{-1}(p)\}\cap \Omega_1\neq\{p\}$,
			$\mathcal{A}(p)$ is the basin of attraction of $p$, $\Omega_i (i=1, 2, \cdots)$ are the connected components of $\mathcal{A}(p)$. Then there is a constant $\tilde{C}$ so that for every point $z_0$ inside any $\Omega_i$, there exists a point $q\in \cup_k f^{-k}(p)$ inside $\Omega_i$ such that $d_{\Omega_i}(z_0, q)\leq \tilde{C}$, where $d_{\Omega_i}$ is the Kobayashi distance on $\Omega_i.$  
		\end{thm}

		Theorem \ref{thmA} shows that in an attracting basin of a complex polynomial, the backward orbit of the attracting fixed point either is the point itself or accumulates at the boundary of all the components of the basin in such a way that all points of the basin lie within a uniformly bounded distance of the backward orbit, measured with respect to the Kobayashi metric. This is an interesting and innovative problem and result. There are no publications by other researchers.

		However,  Hu \cite{RefH2} proved that Theorem \ref{thmA} is no longer valid for parabolic basins of polynomials in one dimension. This is a more interesting and surprising result.
		
	Compared with one dimension, there are very interesting results about dynamics inside attracting basins in higher dimensions.
	Let $F(z, w)=(P(z), Q(w))$, where $P(z), Q(w)$ are two polynomials of degree $m_1, m_2\geq2$ on $\mathbb{C}$,
	and $F^{ n}: \mathbb{C}^2 \rightarrow \mathbb{C}^2$ be its $n$-fold iterate. In complex dynamics, two crucial disjoint invariant sets are associated with $F$, the {\sl Julia set} and the {\sl Fatou set} \cite{RefM}, which partition $\mathbb{C}^2$.
	The Fatou set of $F$ is defined as the largest open set where the family of iterates is locally normal. In other words, for any point $(z, w)\in \mathbb{C}^2$, there exists some neighborhood $U$ of $(z, w)$ so that the sequence of iterates of the map restricted to $U$ forms a normal family, so the iterates are well behaved. 
	The complement of the Fatou set is called the Julia set. 
	
  The classification of Fatou components in one dimension was completed in the '80s when Sullivan \cite{RefS} proved that every Fatou component of a rational map is preperiodic, i.e., there are $n,m \in \mathbb{N}$ such that $f^{n+m}(\Omega)=f^m(\Omega)$. For more details and results, we refer the reader to \cite{RefB, RefCG, RefM}. In higher dimensions, the dynamics is quite different and more fruitful than in $\mathbb{C}.$ We refer the reader to \cite{RefL, RefFS1, RefFS3, RefFS2} for more details and results.
	
	The connected components of the Fatou set of $F$ are called {\sl Fatou components}. 
	A Fatou component $\Omega\subset \mathbb{C}^2$ of $F$ is {\sl invariant} if $F(\Omega)=\Omega$. 
	For $(z, w)\in \mathbb{C}^2$, the set $\{(z_n, w_n)\}=\{(z_1, w_1)=F(z_0, w_0), (z_2, w_2)=F^2(z_0, w_0), \cdots\}$ is called the orbit of the point $(z, w)=(z_0, w_0)$. 
	If $(z_N, w_N)=(z_0, w_0)$ for some integer $N$, we say that $(z_0, w_0)$ is a periodic point of $F$ of periodic $N$.
	If $N=1$, then $(z_0, w_0)$ is a fixed point of $F.$ 
	There have been few detailed studies until now of the more precise behavior of orbits inside the Fatou set. 
	For example, let $\mathcal{A}(z', w'):=\{(z, w)\in \mathbb{C}^2; F^n(z, w)\rightarrow (z', w')\}$ be the {\sl basin of attraction} of an attracting fixed point $(z', w')$. 
	One can ask when $(z_0, w_0)$ is close to $\partial\mathcal{A}(z', w')$, what orbits $\{(z_n, w_n)\}$ of $(z_0, w_0)$ going from $(z_0, w_0)$ to near the attracting fixed point $(z', w')$ look like, or how many iterations are needed. 
	
	In this paper, we will investigate how the orbits behave inside the attracting basin of $F(z, w)$ in $\mathbb{C}^2$. We obtain the following, Theorem \ref{the2} in section 3:

	\begin{thm*}{\bf \ref{the2}.}
		Suppose $F(z, w)=(P(z), Q(w)),$ where $P(z), Q(w)$ are two polynomials of degree $m_1, m_2\geq2$ on $\mathbb{C},$
		$P(0)=Q(0)=0,$ and $0< |P'(0)|, |Q'(0)|<1.$  Let $\Omega$ be the immediate attracting basin of $F(z, w)$. 
		Then there is a constant $C$ such that for every point $(z_0, w_0)\in \Omega$, there exists a point $(\tilde{z}, \tilde{w})\in \cup_k F^{-k}(0, 0),  k\geq0$ so that $d_\Omega\big((z_0, w_0), (\tilde{z}, \tilde{w})\big)\leq C,  d_\Omega$ is the Kobayashi distance on $\Omega$. 
	\end{thm*}

	However, Theorem \ref{the2} is not valid for any of the following cases:
	
	Let $F(z, w)=(P(z), Q(w)),$ where $P(z), Q(w)$ are two polynomials of degree $m_1, m_2\geq2$ on $\mathbb{C}$ and $P(0)=Q(0)=0$.
	\begin{itemize} [itemsep=0pt]
		\item[(1)] $P(z)=z^{m_1},  Q(w)=w^{m_2}$;
		\item[(2)] $P(z)=z^m, 0<|Q'(0)|<1, $ i.e., $P'(0)=0;$
		\item[(3)] $P(z)=z^m, Q'(0)=1, $  i.e., $P'(0)=0;$
		\item[(4)] $0<|P'(0)|<1, Q'(0)=1;$
		\item[(5)] $P'(0)=Q'(0)=1$. 
	\end{itemize}
	
	Polynomial skew products, for example, \cite{RefJM, RefJ}, have been useful test cases for
	complex dynamics in two dimensions. This allows one to use one variable results in one of the variables:
	Suppose $F$ is a polynomial skew product, $F(z, w)=(P(z), Q(z, w)),$ where $P(z), Q(z, w)$ are two polynomials of degree $m_1, m_2\geq2$ on $\mathbb{C}$ and $P(0)=0, Q(0,0)=0$.
We will prove that, in the following cases in section 4, Theorem \ref{the2} fails as well: 
	
	\begin{itemize} [itemsep=0pt]
		\item[(1)] $P(z)=z^2, Q(z, w)=w^2+az, a\in\mathbb{C};$ 
		\item[(2)] $P(z)=az+z^2, Q(z, w)=w^2+cw+bz, 0<|a|, |b|, |c|<<1$ and $ |a|>>|c|, |a|>>|b|, |c|>>|ab|$.
	\end{itemize}

	\section*{Acknowledgement}

I appreciate my advisor John Erik Forn\ae ss very much for giving me this research problem and for his valuable comments and patient guidance. In addition, thanks very much to the employees at NTNU, Norway, for all their kind help during my visit. At last, I am very grateful to the University of Parma for supporting my Ph.D. study. 


%

	\section{Preliminaries}\label{subsec1}
	First of all, let us recall the definition of the Kobayashi metric, for example, \cite{RefA, RefBGN, RefK, RefW}.
	
	\begin{defn*}\label{def5}
		Let $\hat{\Omega}\subset\mathbb C^n$ be a domain. We choose a point $z\in \hat{\Omega}$ and a vector $\xi$ tangent to $\mathbb C^n$ at the point $z.$ Let $\triangle$ denote the unit disk in the complex plane $\mathbb C$.
		We define the {\em Kobayashi metric}
		$$
		F_{\hat{\Omega}}(z, \xi):=\inf\{\lambda>0 : \exists f: \triangle\stackrel{hol}{\longrightarrow} \hat{\Omega}, f(0)=z, \lambda f'(0)=\xi\}.
		$$

		Let $\gamma: [0, 1]\rightarrow \hat{\Omega}$ be a piecewise smooth curve.
		The {\em Kobayashi length} of $\gamma$ is defined to be 
		$$ L_{\hat{\Omega}} (\gamma)=\int_{\gamma} F_{\hat{\Omega}}(z, \xi) \lvert dz\rvert=\int_{0}^{1}F_{\hat{\Omega}}\big(\gamma(t), \gamma'(t)\big)\lvert \gamma'(t)\rvert dt.$$

		For any two points $z_1$ and $z_2$ in $\hat{\Omega}$, the {\em Kobayashi distance} between $z_1$ and $z_2$ is defined to be 
		$$d_{\hat{\Omega}}(z_1, z_2)=\inf\{L_{\hat{\Omega}} (\gamma): \gamma ~ \text{is a piecewise smooth curve connecting} ~z_1~ \text{and} ~z_2 \}.$$

		Note that $d_{\hat{\Omega}}(z_1, z_2)$ is defined when $z_1$ and $z_2$ are in the same connected component of $\hat{\Omega}.$

		Let $d_E(z_1, z_2)$ denote the Euclidean metric distance for any two points $z_1, z_2\in\triangle.$

	\end{defn*}
	
		\section{ Dynamics of holomorphic polynomials inside Fatou sets of $F(z, w)= (P(z), Q(w))$}

				In this section, we study how precisely the orbit goes inside attracting basins of $F(z, w)=(P(z), Q(w))$, where $P(z), Q(w)$ are two polynomials of degree $m_1, m_2\geq2$ on $\mathbb{C}$. 
				
			\subsection{Dynamics of $F(z, w)=(z^m,  w^m), m\geq2$, i.e., $|P'(0)|=|Q'(0)|=0$}

\begin{thm}\label{the1}
	Let $F(z, w)=(z^m,  w^m), m\geq2$. We choose an arbitrary constant $C>0$ and the point $(\epsilon,\epsilon), 0<\epsilon<<1$.  Then there exists a point $(z_0, w_0)\in \mathcal{A}$ so that for any $(\tilde{z}, \tilde{w})\in \cup_{k}^{\infty}\{F^{-k}(\epsilon, \epsilon)\}, k\geq0$,
	the Kobayashi distance $d_{\mathcal {A}}((z_0, w_0), (\tilde{z}, \tilde{w}))\geq C$. 
\end{thm}

\begin{proof}
	
	We know $(\tilde{z}_k, \tilde{w}_k)=(\epsilon^{1/m^k}, \epsilon^{1/m^k} )$ for some positive integer $k.$ Then $|\tilde{z}|=|\tilde{w}|.$ We take $(z_0, w_0)=(1-\delta, \delta)$, $\delta$ is sufficiently small depending on $C$.  Then 
	\begin{equation*}
		\begin{aligned}
	d_{\triangle\times\triangle}\big((z_0, w_0), (\tilde{z}_k, \tilde{w}_k)\big)&=d_{\triangle\times \triangle}\big((z_0, w_0), (\epsilon^{1/m^k}, \epsilon^{1/m^k})\big)\\
	&=\max\bigg\{d_\triangle(z_0, \tilde{z}_k), d_\triangle(w_0, \tilde{w}_k)\bigg\}\\
	&=\max\bigg\{d_\triangle(0, \frac{z_0-\tilde{z}_k}{1-z_0\overline{\tilde{z}_k}}), d_\triangle(0, \frac{w_0-\tilde{w}_k}{1-w_0\overline{\tilde{w}_k}})\bigg\}\\
	&=\max\bigg\{\ln \frac{1+\lvert\frac{z_0-\tilde{z}_k}{1-z_0\overline{\tilde{z}_k}}\rvert}{1-\lvert\frac{z_0-\tilde{z}_k}{1-z_0\overline{\tilde{z}_k}}\rvert}, \ln \frac{1+\lvert\frac{w_0-\tilde{w}_k}{1-w_0\overline{\tilde{w}_k}}\rvert}{1-\lvert\frac{w_0-\tilde{w}_k}{1-w_0\overline{\tilde{w}_k}}\rvert}\bigg\}\\
	&=\max\bigg\{\ln \frac{1+\lvert\frac{1-\delta-\epsilon^{1/m^k}}{1-(1-\delta)\epsilon^{1/m^k}}\rvert}{1-\lvert\frac{1-\delta-\epsilon^{1/m^k}}{1-(1-\delta)\epsilon^{1/m^k}}\rvert}, \ln \frac{1+\lvert\frac{\delta-\epsilon^{1/m^k}}{1-\delta\epsilon^{1/m^k}}\rvert}{1-\lvert\frac{\delta-\epsilon^{1/m^k}}{1-\delta\epsilon^{1/m^k}}\rvert}\bigg\}.
		\end{aligned}
	\end{equation*}
Hence there are two cases in the following:

Case 1:	Notice that $d_{\triangle\times\triangle}\big((z_0, w_0), (\tilde{z}_k, \tilde{w}_k)\big)\geq\ln \frac{1+\lvert\frac{\delta-\epsilon^{1/m^k}}{1-\delta\epsilon^{1/m^k}}\rvert}{1-\lvert\frac{\delta-\epsilon^{1/m^k}}{1-\delta\epsilon^{1/m^k}}\rvert}\geq\ln\frac{1}{1-\lvert\frac{\delta-\epsilon^{1/m^k}}{1-\delta\epsilon^{1/m^k}}\rvert}\rightarrow\infty$, when $k\rightarrow\infty.$ Hence $d_{\triangle\times\triangle}\big((z_0, w_0), (\tilde{z}_k, \tilde{w}_k)\big)\geq C$ if $k>k_0$ for some integer $k_0$ and $\delta$ is smaller than $1/2$.

Case 2: Also, $d_{\triangle\times\triangle}\big((z_0, w_0), (\tilde{z}_k, \tilde{w}_k)\big)\geq\ln \frac{1+\lvert\frac{1-\delta-\epsilon^{1/m^k}}{1-(1-\delta)\epsilon^{1/m^k}}\rvert}{1-\lvert\frac{1-\delta-\epsilon^{1/m^k}}{1-(1-\delta)\epsilon^{1/m^k}}\rvert}\rightarrow\infty$ when $\delta\rightarrow0$ for each fixed $k.$ Hence $d_{\triangle\times\triangle}\big((z_0, w_0), (\tilde{z}_k, \tilde{w}_k)\big)\geq C$ if $k\leq k_0$ and $\delta$ small enough.
\end{proof}

\subsection{Dynamics of $F(z, w)=(P(z), Q(w)), P(0)=Q(0)=0, 0<|P'(0)|, |Q'(0)|<1$ }

\begin{thm}\label{the2}
Suppose $F(z, w)=(P(z), Q(w))$, where $P(z), Q(w)$ are two polynomials of degree $m_1, m_2\geq2$ on $\mathbb{C}$,
$P(0)=Q(0)=0,$ and $0<|P'(0)|, |Q'(0)|<1.$  Let $\Omega$ be the immediate attracting basin of $F(z, w)$. 
Then there is a constant $C$ such that for every point $(z_0, w_0)\in \Omega$, there exists a point $(\tilde{z}, \tilde{w})\in \cup_k F^{-k}(0, 0), k\geq0$ so that $d_\Omega\big((z_0, w_0), (\tilde{z}, \tilde{w})\big)\leq C,  d_\Omega$ is the Kobayashi distance on $\Omega$. 
\end{thm}

\begin{proof}
	Let the immediate basin of attraction of $P(z), Q(w)$ be denoted by $\Omega_P, \Omega_Q$, respectively. Then $\Omega=\Omega_P\times\Omega_Q.$ 
	
	By the conclusion of Theorem 2.9 in \cite{RefH1}, we know that there is a constant $C_P(C_Q)$ such that for every point $z_0(w_0)\in \Omega_P(\Omega_Q)$, there exists a point $\tilde{z}(\tilde{w})\in \cup_k P^{-k}(0)(\cup_k Q^{-k'}(0)), k, k'\geq0$ so that $d_{\Omega_P}(z_0, \tilde{z})\leq C_P,  (d_{\Omega_Q}(w_0, \tilde{w})\leq C_Q),$ where $d_{\Omega_P}(d_{\Omega_Q})$ is the Kobayashi distance on $\Omega_P (\Omega_Q)$. Suppose $K=\max\{k, k'\}$ and $C=\max\{C_P, C_Q\}$, then $\tilde{z}\in\cup_k P^{-K}(0), \tilde{w}\in \cup_k Q^{-K}(0).$ Hence $(\tilde{z}, \tilde{w})\in \cup_K F^{-K}(0, 0)\subset\Omega. $ Therefore, $d_\Omega((z_0, w_0), (\tilde{z}, \tilde{w}))\leq C,  d_\Omega$ is the Kobayashi distance on $\Omega$. 
\end{proof} 
	
	\subsection{Dynamics of $F(z, w)=(P(z), Q(w)), P(0)= Q(0)=0$}
		\begin{itemize}
			\item[(1)] $P(z)=z^m, 0<|Q'(0)|<1, $ i.e., $|P'(0)|=0;$
			\item[(2)] $P(z)=z^m, Q'(0)=1, $  i.e., $|P'(0)|=0;$
			\item[(3)] $0<|P'(0)|<1, Q'(0)=1;$
			\item[(4)] $P'(0)=Q'(0)=1.$ 
		\end{itemize}
	
	\begin{thm}\label{the3}
	Suppose $F(z, w)=(P(z), Q(w))$, where $P(z), Q(w)$ are two polynomials of degree $m_1, m_2\geq2$ on $\mathbb{C}$, $P(0)=Q(0)=0$, and the module of the derivative of $|P'(0)|, |Q'(0)|$ is any one of the above four. Let $\Omega$ be the immediate attracting basin of $F(z, w)$. 
	 Then there exists a point $(z_0, w_0)\in \Omega$ so that for any $(\tilde{z}, \tilde{w})$ inside the preimage set under $\{F^{-k}\}$ of the fixed point $(0, 0)$ or a point very close to $(0, 0),$ 
	the Kobayashi distance $d_{\Omega}((z_0, w_0), (\tilde{z}, \tilde{w}))\geq C$. 
\end{thm}
	


\begin{proof}
	If $|P'(0)|=0, 0<|Q'(0)|<1,$ by Theorem \ref{the1} and \ref{the2}, we know $\Omega=\triangle \times\Omega_Q.$ We choose $z_0=1-\delta$, $\delta$ small enough.
	Then we know that $\tilde{z}=\epsilon^{1/m^k}$ and 
	\begin{equation*}
	\begin{aligned}
d_\Omega((z_0, w_0), (\tilde{z}, \tilde{w}))
&=\max\big(d_\triangle(z_0, \tilde{z}), d_{\Omega_Q}(w_0, \tilde{w})\big)\\
&\geq d_\triangle(z_0, \tilde{z})\\
&=\ln \frac{1+\lvert\frac{1-\delta-\epsilon^{1/m^k}}{1-(1-\delta)\epsilon^{1/m^k}}\rvert}{1-\lvert\frac{1-\delta-\epsilon^{1/m^k}}{1-(1-\delta)\epsilon^{1/m^k}}\rvert}\\
&\rightarrow\infty
		\end{aligned}
	\end{equation*}
as $\delta\rightarrow0.$
Then the proof is done. 

If $P(z)=z^m, Q'(0)=1$ or $0<|P'(0)|<1, Q'(0)=1$, by the result in the paper \cite{RefH2}, we know this theorem is valid since $d_{\Omega_Q}(w_0, \tilde{w})$ is unbounded. Hence the same reason for $P'(0)=Q'(0)=1.$


	\end{proof}

		\section{ Dynamics of 	Polynomial skew products inside attracting basins of $F(z, w)= (P(z), Q(z,w))$}

	Polynomial skew products have been useful test cases for
complex dynamics in two dimensions. This allows one to use one variable results in one of the variables. 
However, for $F(z, w)= (P(z), Q(z,w))$, we can not calculate the Kobayashi distance $d_\Omega((z_0, w_0), (\tilde{z}, \tilde{w}))$ by analyzing $d_{\Omega_P}(z_0, \tilde{z}), d_{\Omega_Q}(w_0, \tilde{w})$ separately as for $F(z, w)=(P(z), Q(w))$ since $\Omega_Q$ also depends on the $z$-coordinate.
In this section, we consider the dynamics of $F(z, w)$ near $(0, 0)$ and study two simple cases. 
			\subsection{Dynamics of $F(z, w)=(z^2,  w^2+az)$}

\begin{thm}\label{thm5}
	Suppose $F(z, w)=(z^2, w^2+az)$, $a\neq0$. Let $\Omega$ be the immediate attracting basin of $(0, 0)$ and $0<\epsilon<<1$. We choose an arbitrary constant $C>0.$ 
Then there exists a point $(z_0, w_0)\in \Omega$ so that for any $(\tilde{z}, \tilde{w})\in \cup_{k}^{\infty}\{F^{-k}(\epsilon, 0)\}, k\geq0$,
the Kobayashi distance $d_{\Omega}((z_0, w_0), (\tilde{z}, \tilde{w}))\geq C$.
\end{thm}

\begin{proof}
We choose $z_0=0.$ Note that the projection map $\pi :\Omega\rightarrow \triangle, \pi(z, w)=z$ is distance decreasing in the Kobayashi metric. In addition, we know that the $z$-coordinate of any point in $F^{-k}(\epsilon, 0)$ approaches to $\partial\triangle$ as $k\rightarrow\infty.$ Hence there is an $l$ so that if $k>l,$ then $d_\Omega((0, w_0), F^{-k}(\epsilon, 0))\geq C$ for any $w_0.$ Let $(z_j, w_j), j=1, \cdots, N$ be the points in $\{F^{-k}(\epsilon, 0)\}_{k\leq l}.$ Next, we want to show that there is a $w_0$ so that $(0, w_0)\in\Omega$ and $d_\Omega\big((0, w_0), (z_j, w_j)\big)\geq C$ for any $j=1, \cdots, N.$ 

We first deal with a small $a$ and then with general $a$. When $a$ is small, we have the following lemma. 
\begin{lem}\label{lem1}
	Let $D:=\{(z,w); |z|<1, |w|<3/4\}$ be a bidisc. If $|a|<3/16$, then $D\subset\Omega.$
\end{lem}
\begin{proof}
	If $|w|<3/4$, then $|w^2|<9/16.$ Hence we have $|w^2+az|<9/16+|a|<3/4.$
\end{proof}

Now we continue to prove Theorem \ref{thm5}. Let $f(z)\equiv w_0=2/3$ for $|z|<1.$ Then $\gamma_0=(z, f(z))$ is a graph over $|z|<1$ inside $\Omega$.
Then $F^{-1}(\gamma_0)=F^{-1}(z, f(z))=(\sqrt{z}, \sqrt{f(z)- a \sqrt{z}}).$ Let us choose $\gamma_1=(\sqrt{z}, \sqrt{f(z)-a\sqrt{z}}):=(z, \sqrt{f(z^2)-az})=(z, f_1(z)).$ By induction, we can have $\gamma_2=(z, \sqrt{f_1(z^2)-az})=(z, f_2(z));\dots; \gamma_n=(z, f_n(z))$, here we always choose  $f_n(z)$ so that $f_n(0)>0$. Note that inductively $f_n(z)\geq 2/3,$ hence $f_n(z^2)-az$ never has any zeros, which means any branches of $f_n$ cannot meet at some points of $z\in D,$ i.e., none of any two graphs of $\gamma_n$ intersect each other in $D$. Then  $\lim_{n\rightarrow\infty}\gamma_n\subset\partial\Omega$ since $\gamma_0$ does not go through $(0,0)$ and $\gamma_0\subset\Omega$ so that the backward orbits of $\gamma_  0$ converge to $\partial\Omega.$ By Montel's theorem, there is a convergent subsequence $f_n(z)\rightarrow f_\infty(z)$, then $(z, f_\infty(z))\subseteq \partial\Omega$ and $f_\infty(0)=1.$ This implies that for any $z$, $\gamma_\infty(z)\in\partial\Omega_z$ for every slice $\Omega_z.$ 

Let $h(z, w)=w-f_\infty(z)$, then $h(z, w)$ is holomorphic on $\Omega.$ 
And we know that $h(0, w)=w-1$, then $\lim_{w\rightarrow1}h(0, w)=0.$ Hence $h(z, w)$ is vanishing at $(0, 1)$ and $h(\Omega) \subset \triangle(0, R)\setminus\{0\}$ for some constant radius $R>1$ since $\Omega$ is a bounded set. 

Let us recall the Kobayashi metric on the punctured disk, see example 2.8 in \cite{RefM}. The universal covering surface of $\triangle(0, R)\setminus\{0\}$ can be identified with the left half-plane $\{w=u+iv; u<0\}$ under the exponential map $ w\mapsto z=Re^w\in\triangle(0, R)\setminus\{0\}$ with $dw=\frac{dz}{z}$. Hence the Kobayashi metric $|\frac{dw}{u}|$ on the left half-plane corresponds to the metric $\bigg|\frac{dz}{r \ln \frac{r}{R}}\bigg|$ on the punctured disk $\triangle(0, R)\setminus\{0\}$, where $r=|z|$ and $u=\ln \frac{r}{R}.$

Then let $x:=h(0, w_0):=\lim_{w\rightarrow1}h(0, w)$ and $y:=h(z_j, w_j)$ for $j=1, \cdots, N.$
Then 
	\begin{equation*}
	\begin{aligned}
d_\Omega((0, w_0), (z_j, w_j))&\geq d_{\triangle(0, R)\setminus\{0\}}(x, y)\\
&=\bigg|\int_{x}^{y}\frac{1}{|z| \ln \frac{|z|}{R}}dz\bigg|\\
&=|\ln(|\ln y/R|)-\ln(|\ln x/R|)|\\
&\rightarrow\infty,
	\end{aligned}
\end{equation*}
as $x\rightarrow0.$  

For general $a$, although we cannot choose a cylinder as in Lemma \ref{lem1}, we can first choose $D_0:=\{(z,w); |z|<\eta<<1, |w|<3/4\}$. Let $\hat{f}(z)=\hat{w}_0=2/3$ for $|z|<\eta$, then $\hat{\gamma}_0=(z, \hat{f}(z))$ is a graph over $|z|<\eta.$ We have $F^{-1}(\hat{\gamma}_0)=F^{-1}(z, \hat{w}_0)=(\sqrt{z}, \sqrt{\hat{w}^2_0- a\sqrt{z}}).$ Then we choose $\hat{f}_1:=\sqrt{\hat{f}(z^2)-az}$ restricted to $|z|<\eta.$ Inductively, $\hat{f}_{n+1}(z)=\sqrt{\hat{f}_n(z^2)-az}$ is restricted to $|z|<\eta$ as well. Then one gets $\hat{f}_n(z)\rightarrow \hat{f}_\infty(z)$, then $(z, \hat{f}_\infty(z))\subseteq \partial\Omega$ and $\hat{f}_\infty(0)=1.$

Second, let $D_1:=\{(z,w); |z|<\sqrt{\eta}, w\in\mathbb{C}\}$. We choose $g_0(z)=w'_0=\hat{f}_\infty(z)$ where $|z|<\eta.$  And $\gamma'_0=(z, g_0(z))$ is a graph over $|z|<\eta.$ Then we have $F^{-1}(z, w'_0):=(z, g_1(z))$ for $|z|<\eta.$ However, there are two cases:

(1) If $g_0-az\neq0$ for $|z|<\sqrt{\eta}.$ Then there are two solutions of $g_1$, we denote them by $g_{1,1}:=\sqrt{w'^2_0- a\sqrt{z}};  g_{1,2}:=-\sqrt{w'^2_0- a\sqrt{z}}, |z|<\sqrt{\eta}$. We let $g_1$ be one of them and $D_2:=\{(z,w); |z|<\sqrt{\eta}, w\in\mathbb{C}\}.$ 

(2) If $g_0-az=0$ has one or more zeros on $|z|<\sqrt{\eta}.$ We let $g_1$ denote the multivalued function. We repeat this for $g_2(z)=\sqrt{g_1(z^2)-az}$, etc. We continue this process until we obtain a multivalued function $g_n(z)$ well defined on $|z|<\eta^{1/2^n}$, here $n$ will be determined below. Hence, $g_n(z)$ will have at most $2^n$ sheets. 
Meanwhile, we let
 $D_{n+2}:=\{(z,w); |z|<\eta^{1/2^n}, w\in\mathbb{C}\}.$   
 
 Then one gets $g_{n}(z)\rightarrow g_\infty(z)$, then $(z, \lim_{n\rightarrow\infty}g_{n}(z))\subseteq \partial\Omega$ and $\lim_{n\rightarrow\infty}g_{n}(0)=1.$
However, we cannot simply choose $\hat{h}(z, w)=w-g_\infty(z)$. The reason is there are $2^n$ sheets of $g_n(z)$ for every integer $n$, and it is possible that some of them meet at some $z\in D_{n+1}$, eg. $g_{n, 1}(z')=g_{n, 100}(z')=0$ for some $z'\in\{z; |z|<\eta^{1/2^n}\}$, i.e., $g_{n-1}(z^2)-az$ has zeroes in $D_{n}$. 
Hence we let $\hat{h}(z, w):=\prod_{i=1, \cdots, 2^n}\big(w-g_{i}(z)\big).$ Then $h$ is holomorphic on $D_{n+1}$ and $\lim _{w\rightarrow1}\hat{h}(0, w)=0.$ Thus, $\hat{h}$ vanishes at $(0, 1)$ and $\hat{h}(\Omega) \subset \triangle(0, R)\setminus\{0\}$ for some constant radius $R>1$ since $\Omega$ is a bounded set. In addition, we can choose $n$ big enough so that all finitely many $(z_j, w_j)$ are inside $D_{n+1}.$ 

Then let $x':=\hat{h}(0, w_0):=\lim_{w\rightarrow1}\hat{h}(0, w)$ and $y':=\hat{h}(z_j, w_j)$ for $j=1, \cdots, N.$
Then 
\begin{equation*}
	\begin{aligned}
		d_{D_{n+1}}((0, w_0), (z_j, w_j))&\geq d_{\triangle(0, R)\setminus\{0\}}(x', y')\\
		&=\bigg|\int_{x'}^{y'}\frac{1}{|z| \ln \frac{|z|}{R}}dz\bigg|\\
		&=|\ln(|\ln y'/R|)-\ln(|\ln x'/R|)|\\
		&\rightarrow\infty,
	\end{aligned}
\end{equation*}
as $x'\rightarrow0.$

In the end, we still need to show that for any two points $(0, w_0), (z_j, w_j)$ inside $D_{n+1}$, the Kobayashi distance $d_{D_{m}}((0, w_0), (z_j, w_j))$ is approximately equal to $d_\Omega((0, w_0), (z_j, w_j))$ as long as $D_{n+1}\subset\subset D_m$, and $D_m$ is very close to $\Omega.$ We prove the following localization result for the Kobayashi metric (see \cite{RefBGN}).

\begin{lem}
	For $0<s<1,$ let $\Omega_s=\{(z, w)\in\Omega; |z|<s\}.$ Fix $0<r<1,$ let $0<c<1$. Then there exists an $R$ $ (r<R<1)$ so that for every $p\in\Omega_r$ and $\xi$, we have 
 $$c K_\Omega(p, \xi)\geq K_{\Omega_R}(p, \xi)\geq K_\Omega(p, \xi).$$
\end{lem}
	
\begin{proof}
	By Definition in section 2, we know 
		$$F_{\Omega_r}(p, \xi):=\inf\{\lambda>0 : \exists f: \triangle\stackrel{hol}{\longrightarrow} \Omega_r, f(0)=p, \lambda f'(0)=\xi\}.$$
		We choose $r<R<1,$
		$$F_{\Omega_R}(p, \xi):=\inf\{\mu>0 : \exists g: \triangle\stackrel{hol}{\longrightarrow} \Omega_r, g(0)=p, \lambda g'(0)=\xi\}.$$
		Let $g :\triangle\rightarrow\triangle$ and $g(0)=p, g'(0)=\mu\xi.$ By Schwarz Lemma, we know that $|g(rz)|<|rz|<r$ for $|z|<1.$ Then we choose $f(z)=g(rz)\in\Omega_r$ with $f'(0)=cg'(0.)$ 
		Therefore, we have  
		$$c K_\Omega(p, \xi)\geq K_{\Omega_R}(p, \xi)\geq K_\Omega(p, \xi).$$
\end{proof}
Hence, $d_\Omega((0, w_0), (z_j, w_j)) \approx d_{D_{n+1}}((0, w_0), (z_j, w_j))\rightarrow\infty$ for general $a.$ Thus, there exists a point $(z_0, w_0)=(0, 1-\delta)\in \Omega,$ where $\delta\rightarrow0$, so that for any $(\tilde{z}, \tilde{w})\in \cup_{k}^{\infty}\{F^{-k}(\epsilon, 0)\}, k\geq0$,
the Kobayashi distance $d_{\Omega}((z_0, w_0), (\tilde{z}, \tilde{w}))\geq C$.

	\end{proof}

			\subsection{Dynamics of $F(z, w)=(z^2+az,  w^2+cw+bz),  0<|a|, |b|, |c|<<1$}
			
				\begin{thm}\label{thm2}
				Suppose $F(z, w)=(az+z^2, w^2+cw+bz), 0<|a|, |b|, |c|<<1,$ and $ |a|>>|c|, |a|>>|b|, |c|>>|ab|$. 
				Let $\Omega$ be the immediate attracting basin of $(0, 0)$. We choose an arbitrary constant $C>0.$ 
				Then there exists a point $(z_0, w_0)\in \Omega$ so that for any $(\tilde{z}, \tilde{w})\in \cup_{k}^{\infty}\{F^{-k}(0, 0)\}, k\geq0$,
				the Kobayashi distance $d_{\Omega}((z_0, w_0), (\tilde{z}, \tilde{w}))\geq C$.

			\end{thm}
			
				Note that if we first fix $a,c$ and let $b$ be chosen smaller and smaller, then Theorem \ref{the2} always fails, but in the limit case when $b=0$, Theorem \ref{the2} is valid. This shows that the situation is very unstable.
			
			\begin{proof}
				To prove this theorem, we first prove the following lemma:
				\begin{lem}\label{lem1}
					Let $\Omega_{2/3}=\{(z, w), z\in\Omega_P, |w|<2/3\},$ then $F(\Omega_{2/3})\subseteq\Omega_{2/3}.$ Moreover, $\Omega_{2/3}\subseteq\Omega.$

				\end{lem}
				
				\begin{proof}
				 We know $\Omega_P\subseteq\{z, |z|<2\}.$ If $|w|<2/3,$ we obtain
				$|w^2+cw+bz|\leq|w^2|+|c||w|+|b||z|<4/9+2c/3+|b||z|<4/9+|c|+2|b|<2/3$ for $0<|c|, |b|<<1.$ Thus, $F(\Omega_{2/3})\subseteq\Omega_{2/3}.$ 
				
				Let $(z_1, w_1)\in\Omega_{2/3}$ and $(z_n, w_n)$ be the orbit of $(z_1, w_1).$ We know $z_n\rightarrow0,$ and $|w_n|<2/3$ since $F(\Omega_{2/3})\subseteq\Omega_{2/3}.$ Then  $|w_{n+1}|\leq|w_n|^2+|c||w_n|+|b||z_n|\leq(2/3+|c|)|w_n|+|b||z_n|.$ Hence $w_n\rightarrow0.$ Therefore, $\Omega_{2/3}\subseteq\Omega.$ 
				\end{proof}

				\begin{lem}
					Let $\Omega_z$ be the slice of $\Omega$ at $z$ and $\Omega_{z, 2/3}=\{(z, w)\in\Omega_z, |w|<2/3\}.$ Then, each $\Omega_z$ is connected and simply connected.
				\end{lem}
				
				\begin{proof}
				
				Since $Q(w)=w^2+cw+bz$ is a two-to-one function, every point $w\in\Omega_{P(z)}$ has two preimages inside $\Omega_z$ counting multiplicity. Hence $F: \Omega_z\rightarrow\Omega_{P(z)}$ is a double covering and it has a critical point $w=-\frac{c}{2}$ in the $w$-coordinate.
			
				Suppose there exists at least two disjoint connected components $\Omega^1_z, \Omega^2_z$ inside $\Omega_z$, and  $(z, 0)\in\Omega^1_z.$ Then $\Omega_{z, 2/3}\subseteq\Omega^1_z.$
				By Lemma \ref{lem1}, we know $F( \Omega_{z, 2/3})\subseteq\Omega_{P(z), 2/3}.$ In addition, $F$ sends $\Omega_z$ to $\Omega_{P(z)},$ then we know
				$F(\Omega^1_z)\subseteq\Omega_{P(z)}.$ 
			
			Furthermore, we know that every point inside $\Omega_{P(z), 2/3}$ has two preimages inside $\Omega^1_z$ since $F$ is a double covering, and it has a critical point $w=-\frac{c}{2}$, which is very close to $0$ in the $w$-coordinate. Hence $F(\Omega_z^2)\cap \Omega_{P(z), 2/3}=\emptyset.$ Inductively, we know $ F^n(\Omega_z^2)\cap\Omega_{P^n(z), 2/3}=\emptyset.$ Thus, $F^n(\Omega^2_z)$ cannot converge to $0.$ Therefore, $\Omega^2_z$ is not inside $\Omega_z.$ 
			
			By the maximum principle, we know that $\Omega_z$ is simply connected.


				\end{proof}
			
			
				\begin{lem}
			Let $z_N$ be a preimage of $-a\in P^{-1}(0),$ i.e., $z_N\in P^{-N}(-a)$, where $N$ is large enough. Then we have 
			$(0, 0)\notin F^n(\Omega_{z_N, 2/3})$ for any integer $n\geq1.$
			\end{lem}
			
			\begin{proof}

				Let us take a point $(z_N, w_N)\in\Omega_{z_N, 2/3}$ where $N$ is sufficiently large, then $F(z_N, w_N)=(z_{N-1}, w_{N-1})\in\Omega_{z_{N-1}, 1/2}, F^2(z_N, w_N)=(z_{N-2}, w_{N-2})\in\Omega_{z_{N-2}, 1/3}, F^3(z_N, w_N)=(z_{N-3}, w_{N-3})\in\Omega_{z_{N-3}, 1/4}.$ If $4|b|<|w|<1/4,$ we know that
				$$ |w^2+cw+2b|<|w|(|w|+|c|+1/2)<7/8|w|.$$
				This shows $w$ will shrink to $0$ very fast. Thus, for some uniformly large $L>>4$, we will have $|w_{N-l}|<4|b|$ for all $L\leq l \leq N.$ Then inductively, $F^{N-1}(z_N, w_N )=(-a, w_1)\in\Omega_{z_1, 4b}, i.e., |w_1|\leq 4|b|.$ And $F^N(z_N, w_N)=(0, w_0)=(0, w_1^2+cw_1+bz_1).$
				Then $$0<\frac{1}{2}|ab|\leq|ab|-16|b|^2-4|cb| \leq|w_0|=|w_1^2+cw_1+bz_1|<16|b|^2+4|cb|+|ab|\leq2|ab|<<|c|$$ 
				since $|a|>>|c|, |a|>>|b|, |c|>>|ab|.$

		 Therefore, $(0, 0)\notin F^n(z_N, \Omega_{2/3})$ for any $n\leq N.$ However, for $n>N,$ we use that $F$ restricted to the $w$-axis is just $w\rightarrow w^2+cw\approx cw$, so $w_n$ goes to $0$ but never lands on $0$.

				\end{proof}

			\begin{lem}\label{lem2}
				Let $\Omega_n=F^{-n}(\Omega_{2/3}).$ Then for $(\tilde{ z}_n, \tilde{w}_n)\in\Omega_n,$ the Euclidean distance $d_E(\tilde{w}_n, \partial\Omega_{\tilde{z}_n})\rightarrow 0$ in $\Omega_{\tilde{z}_n}$ when $n\rightarrow\infty.$
			\end{lem}
			
			\begin{proof}
				Suppose that for some $\epsilon>0,$ there exists arbitrarily large $N_1$ such that, in $\Omega_{\tilde{z}_{N_1}},$ the Euclidean distance $d_E(\tilde{w}_{N_1}, \partial\Omega_{\tilde{z}_{N_1}})>\epsilon.$ Then since $\big|\frac{\partial Q}{\partial w}\big|>2|w|-|c|>\frac{7}{6}>1$ for $|w|>\frac{2}{3},$ it follows that the distance  $d_E(\tilde{w}_{N_1-1}, \partial\Omega_{\tilde{z}_{N_1-1}})>\frac{7}{6}\epsilon.$ Repeating this  for large $l$ times, we have $d_E(\tilde{w}_{N_1-l}, \partial\Omega_{\tilde{z}_{N_1-l}})>\frac{7}{6}\epsilon^l\geq 4.$ It will get a contradiction to $d_E(\tilde{w}_{N_1-l}, \partial\Omega_{\tilde{z}_{N_1-l}})$ bounded by $4.$ 
				
			\end{proof}

			\begin{lem}
			Let $D:=\partial\Omega\cap\{(z, w); z\in P(z)\}$. Then $D$ is laminated by holomorphic graphs $w=f_\alpha(z).$ Moreover, the Kobayashi distance $d_\Omega((z_N, 0), (\tilde{ z}, \tilde{w}))\geq C$ for any $(z_N, 0)\in \Omega_{z_N}, (\tilde{ z}, \tilde{w})\in\cup_{k}^{\infty}\{F^{-k}(0, 0)\}\subset\Omega_{z_k}, k\geq0, N\geq N(C).$
		\end{lem}

		\begin{proof}

		The set $\partial\Omega_{2/3}$ is laminated by graphs $w=\frac{2}{3}e^{i\theta}.$ Then we take $F^{-1}\big(w=\{\frac{2}{3}e^{i\theta}\}\big)= \frac{-c\pm \sqrt{c^2-4bz-\frac{8}{3}e^{i\theta}}}{2}:=f^{1, 2}_1.$ It is obvious that $c^2-4bz-\frac{8}{3}e^{i\theta}\neq 0$ since $0<|b|, |c|<<1, |z|<2.$ 
	  Hence $F^{-1}\big(w=\{\frac{2}{3}e^{i\theta}\}\big)$ always has two disjoint preimages. Then we can use $f^{1, 2}_1$ to laminate $F^{-1}(\partial\Omega_{2/3})$. 
	  
	  Inductively, we can use $f^{1, 2}_j$ to laminate $F^{-j}(\partial\Omega_{2/3}), j\geq 2$.
	  We know that  $F^{-(j-1)}(\partial\Omega_{2/3})$ is laminated by $f^{1, 2}_{j-1}$, hence $w=f^{1,2}_{j-1}(z)$ are graphs inside $ F^{-(j-1)}(\partial\Omega_{2/3}).$ 
	  Then we calculate $F^{-1}(w=f^{1, 2}_{j-1}(z)):$ let
	  $$(Z, W)\in F^{-1}(w=f^{1, 2}_j(z)) ~~\text{i.e.,} ~~F(Z, W)\in(w=f^{1, 2}_j(z)).$$ Hence 
	  $$Z^2+aZ=z, W^2+cW+bZ=w,$$ 
	  then 
	  $$W^2+cW+bZ=f_j(z)=f_j(Z^2+aZ),$$ 
	  $$W^2+cW+bZ-f_j(Z^2+aZ)=0,$$ 
	  thus, 
	  $$W=\frac{-c\pm\sqrt{c^2-4bZ+4f_j(Z^2+aZ)}}{2}=f^{1, 2}_{j+1}(Z).$$ 
	  Then let $f_\alpha(z):=\lim_{j\rightarrow\infty}F^{-j}\big(w=\{\frac{2}{3}e^{i\theta}\}\big).$ Therefore, $D$ is laminated by graphs $w=f_\alpha(z).$ 

			Next, we need to show in any different slices $\Omega_{z_N},$ and any $(\tilde{z}, \tilde{w})$ in any slices $\Omega_{z_k}$, we always have $d_\Omega((z_N, 0), (\tilde{ z}, \tilde{w}))\geq C$ as long as $(\tilde{ z}, \tilde{w})\rightarrow\partial\Omega$. Here we choose $N$ sufficiently large. Note that the Kobayashi distance $d_{\Omega_{P}}(z_k, z_N)\geq C$, if $z_N\rightarrow\partial\Omega_P$ for a fixed $k$ (See \cite{RefW}).  So we can assume both $k$ and $N$ are very large. However, by Lemma \ref{lem2}, $\tilde{w}$ is very close to some point denoted by $(\tilde{ z}, \eta)$ in $\partial\Omega_{\tilde{ z}}.$
		
			Let $H:\Omega\rightarrow \triangle(0, R)\setminus\{0\}, H(z, w)=w-f_\alpha(z)$. Here we choose $\alpha$ so that $f_\alpha(\tilde{z})=\eta.$ Then $H(z, w)$ is holomorphic on $\Omega.$ 
			And we know that $H(z, w)=w-f_\alpha(z)$, then $\lim_{w\rightarrow f_\alpha(z)}H(z, w)=0.$ Hence $H(z, w)$ is vanishing at $(\tilde{z}, \eta)$ and $H(\Omega) \subset \triangle(0, R)\setminus\{0\}$ for some constant radius $R>1$ since $\Omega$ is a bounded set. 
			
			Then let $x':=H(z, w)$ and $y':=H(\tilde{ z}, \tilde{w})$. 
			Then 
			\begin{equation*}
				\begin{aligned}
					d_{\Omega}((z, w), (\tilde{ z}, \tilde{w}))&\geq d_{\triangle(0, R)\setminus\{0\}}(x', y')\\
					&=\bigg|\int_{x'}^{y'}\frac{1}{|z| \ln \frac{|z|}{R}}dz\bigg|\\
					&=|\ln(|\ln y'/R|)-\ln(|\ln x'/R|)|\\
					&\rightarrow\infty,
				\end{aligned}
			\end{equation*}
			as $x'\rightarrow0, i. e., w\rightarrow f_\alpha(z).$  
					\end{proof}
				
			

		Now we continue to prove Theorem \ref{thm2} using the same method as in Theorem \ref{thm5}. We take $(z_0, w_0)=(z_N, 0)\in\Omega_{z_N, 2/3}.$ Then $H(\Omega_{2/3})\subset \triangle(0, R)\setminus \{0\}, H(\tilde{z}, \tilde{w})\rightarrow  0$ as $k\rightarrow\infty.$
					
		Therefore, we know that there exists a point $(z_0, w_0)=(z_N, 0)$ so that for any $(\tilde{z}, \tilde{w})\in \cup_{k}^{\infty}\{F^{-k}(0, 0)\}, k\geq0$,
		the Kobayashi distance $d_{\Omega}((z_0, w_0), (\tilde{z}, \tilde{w}))=d_{\Omega}\big((z_N, 0), (P^{-i}(0), Q^{-i}(0))\big)\geq d_{\triangle(0, R)\setminus\{0\}}(x', y')\geq C$ for all $i\in\mathbb{N}$.  
				\end{proof}

\small University of Parma, Department of Mathematical, Physical and Computer Sciences, Parco Area delle Scienze, 53/A, 43124 Parma PR, Italy

		\emph{Email address: mi.hu@unipr.it}
\end{document}